\documentclass[11pt,a4paper]{article}

\usepackage[utf8]{inputenc}
\usepackage[T1]{fontenc}
\usepackage{lmodern}
\usepackage[a4paper,margin=1in]{geometry}
\usepackage{microtype}
\usepackage[numbers,sort&compress]{natbib}
\usepackage{amsmath,amssymb,amsfonts,amsthm,amscd}
\usepackage{mathrsfs}
\usepackage{mathtools}
\usepackage{array}
\usepackage{tikz-cd}
\usepackage{hyperref}

\hypersetup{
    colorlinks=true,
    linkcolor=blue,
    citecolor=blue,
    urlcolor=blue
}

\newtheorem{theorem}{Theorem}[section]
\newtheorem{lemma}[theorem]{Lemma}
\newtheorem{proposition}[theorem]{Proposition}
\newtheorem{corollary}[theorem]{Corollary}
\theoremstyle{definition}
\newtheorem{definition}[theorem]{Definition}
\newtheorem{example}[theorem]{Example}
\theoremstyle{remark}
\newtheorem{remark}[theorem]{Remark}

\newcommand{\CP}{\mathbb{CP}}
\newcommand{\so}{\mathfrak{so}}

\newcommand{\fu}{\mathfrak{u}}

\newcommand{\ft}{\mathfrak t}
\newcommand{\fh}{\mathfrak h}
\newcommand{\fm}{\mathfrak m}
\newcommand{\vol}{\operatorname{vol}}
\newcommand{\Span}{\operatorname{Span}}
\newcommand{\Ad}{\operatorname{Ad}}
\newcommand{\SO}{\operatorname{SO}}

\title{Contact Monopoles in Dimension Five:\\
Rigidity and Homogeneous Classification}

\author{Luis Portilla\\
Department of Mathematics, University of Cauca\\
Popay\'an 190002, Colombia\\
\texttt{luisportilla@unicauca.edu.co}\\
\href{https://orcid.org/0000-0001-7511-2910}{ORCID: 0000-0001-7511-2910}}

\date{}

\begin{document}

\maketitle

\begin{abstract}
We introduce contact monopoles on contact metric $5$-manifolds and study their basic geometric properties. We construct natural examples from Boothby--Wang fibrations over Kähler surfaces and relate the equation to contact instantons on Sasakian $5$-manifolds. In particular, we show that anti-self-dual contact instantons with a parallel Higgs field provide contact monopoles. We also prove a compact rigidity result: on compact contact metric $5$-manifolds, every smooth contact monopole has parallel Higgs field and satisfies the reduced curvature equation $F_A\wedge\omega=0$. In the Sasakian case, this relates compact contact monopoles to anti-self-dual contact instantons under the usual transverse type condition. We also explain why the associated deformation complex is not elliptic in the usual sense.

We then focus on unit tangent sphere bundles of oriented Riemannian $3$-manifolds, endowed with their canonical contact metric structures. In this setting, we derive a local form of the contact monopole equation with respect to the canonical adapted frame and prove that pure pullback configurations from the base are rigid. Finally, for the homogeneous Sasakian manifold
$T_1S^3\cong SO(4)/SO(2)$, we use Wang's theorem to reduce the invariant contact monopole equation to finite-dimensional algebraic systems and classify the invariant solutions for homogeneous $U(1)$- and $SO(3)$-bundles. In the abelian case, the classification yields non-flat homogeneous $U(1)$ contact monopoles for nonzero topological charge. In the non-abelian case, all weights except $m=1$ reduce to the maximal torus, while the weight $m=1$ gives a genuinely non-abelian branch of homogeneous $SO(3)$ contact monopoles.
\end{abstract}


\medskip

%---------------------------------------
\section{Introduction}\label{sec:1}
%--------------------------------------
Monopole equations arise naturally in gauge theory as first-order equations
relating the curvature of a connection to the covariant derivative of a Higgs
field. A common way to produce such equations in different dimensions is to use
a distinguished differential form on the underlying manifold. More precisely,
if $M^n$ carries a distinguished form $\Theta$ of degree $n-3$, one may consider
equations of the form
$$
\ast(F_A\wedge \Theta)=\nabla_A\phi.
$$
In dimension $3$, taking $\Theta=1$ gives the classical Bogomolny equation
\cite{bogomolny1976}. In dimension $7$, if $(M^7,\varphi)$ is endowed with a
$G_2$-structure and $\psi=\ast\varphi$ is the coassociative $4$-form, taking
$\Theta=\psi$ gives the $G_2$-monopole equation studied, for instance, in
\cite{oliveira2014monopoles}. The aim of this paper is to study the analogous
construction in dimension $5$ in the setting of contact metric geometry.

Let $(M^5,\eta,\xi,\Phi,g)$ be a contact metric $5$-manifold. The contact
structure provides a natural distinguished transverse $2$-form, namely the
fundamental form $\omega=d\eta$. Thus, for a connection $A$ on a principal
$G$-bundle $P\to M$, where $G$ is a compact Lie group, and for an adjoint-valued
Higgs field $\phi\in\Omega^0(M,\mathfrak g_P)$, we consider the equation
\begin{equation}\label{eq:intro-contact-monopole}
\ast(F_A\wedge\omega)=\nabla_A\phi.
\end{equation}
We call solutions of \eqref{eq:intro-contact-monopole} \emph{contact
monopoles}.  The contact setting has features which distinguish it from the classical Bogomolny and $G_2$ settings. In particular, the map
$$
\alpha\mapsto \ast(\alpha\wedge\omega)
$$
does not detect all components of a $2$-form. This degeneracy affects the
linearized theory and prevents the associated deformation complex from being
elliptic in the usual sense. Moreover, on compact contact metric
$5$-manifolds, the equation exhibits a strong rigidity phenomenon: every smooth
contact monopole has parallel Higgs field, and the equation reduces to the
system
$$
F_A\wedge\omega=0,
\qquad
\nabla_A\phi=0.
$$
Thus the compact smooth theory does not directly lead to a Fredholm moduli
problem analogous to those arising in the classical Bogomolny or
$G_2$-monopole settings. For this reason, in the present paper we focus on the
geometric meaning of the reduced curvature equation, its relation with contact
instantons, and explicit classifications in natural contact metric examples.

Our first class of examples comes from Boothby--Wang fibrations
\cite{BoothbyWang1958,boyer2007sasakian}. If
$S^1\to M\to X$ is a Boothby--Wang fibration over a compact Kähler surface
$(X,\omega_X)$, then the total space carries a natural regular Sasakian
structure with $d\eta$ equal to the pullback of $\omega_X$. In this setting,
we show that pullback connections with parallel Higgs field solve the contact
monopole equation precisely when the curvature on the Kähler surface is
primitive. Thus, the contact monopole equation is naturally related to the
Hermitian Yang--Mills condition on Kähler surfaces. We also provide explicit
examples coming from line bundles over $\mathbb CP^1\times\mathbb CP^1$, and
we discuss a non-example on the Hopf fibration $S^5\to\mathbb CP^2$, showing
that the primitive condition is essential.

We also relate contact monopoles to contact instantons on Sasakian
$5$-manifolds. Contact instantons, in the sense of
\cite{baraglia2016moduli}, are defined by a self-duality condition involving
the contact form. In dimension $5$, anti-self-dual contact instantons have
transverse primitive curvature. Therefore, every anti-self-dual contact
instanton with parallel Higgs field gives a contact monopole. On the other
hand, the compact rigidity result mentioned above implies that every smooth
contact monopole on a compact contact metric $5$-manifold has parallel Higgs
field and satisfies $F_A\wedge\omega=0$. In the Sasakian case, if the curvature
has transverse type $(1,1)$, this reduced curvature equation is precisely the
primitive condition which characterizes anti-self-dual contact instantons
\cite{baraglia2016moduli}. Thus, compact smooth contact monopoles on Sasakian
$5$-manifolds are closely related to contact instantons. Moreover, if the
connection is irreducible, then the Higgs field vanishes.

This rigidity motivates the study of explicit geometric reductions. In the
second part of the paper, we focus on unit tangent sphere bundles $T_1N$ of
oriented Riemannian $3$-manifolds. Such manifolds carry a canonical contact
metric structure \cite{Tashiro1969,Blair2002,CalvarusoUnitTangent}. Using the
horizontal and tangential vertical decomposition determined by the
Levi--Civita connection of the base, we derive a local component form of
\eqref{eq:intro-contact-monopole} with respect to a canonical adapted frame.
This local system allows us to prove a rigidity result for pure pullback
configurations: if both the connection and the Higgs field are obtained by
pulling back data from the base $N$, then the contact monopole equation holds
if and only if the base connection is flat and the Higgs field is parallel.
Hence the contact monopole equation on $T_1N$ is not simply a lift of the
classical Bogomolny equation on $N$.

Finally, we specialize to the round sphere $N=S^3$. In this case,
$$
T_1S^3\cong V_2(\mathbb R^4)\cong SO(4)/SO(2).
$$
The canonical contact metric structure is homogeneous and Sasakian
\cite{CalvarusoPerrone2010}. This allows us to study invariant contact
monopoles using homogeneous bundle theory. We recall the classification of
homogeneous principal bundles over $SO(4)/SO(2)$ in terms of homomorphisms
from the isotropy group into the structure group, and we use Wang's theorem
\cite{wang1958invariant} to describe invariant connections. The contact
monopole equation then reduces to finite-dimensional algebraic systems.

For homogeneous $U(1)$-bundles, we classify invariant contact monopoles and
obtain non-flat solutions for $m\neq0$. In view of the rigidity theorem for
pure pullbacks, these solutions do not arise from pulling back monopole data
from the base $S^3$. We then consider the first non-abelian case, namely
homogeneous $SO(3)$-bundles. These bundles are determined, up to conjugation,
by weights $m\in\mathbb Z_{\geq0}$. We show that the cases $m\neq1$ reduce to
the maximal torus and do not produce genuinely non-abelian homogeneous contact
monopoles. In contrast, the case $m=1$ admits a genuinely non-abelian branch
of invariant solutions. This provides explicit homogeneous non-abelian contact
monopoles in dimension five.

\subsection*{Description of the main results}

In Section \ref{sec:2}, we introduce the contact monopole equation
\eqref{eq:intro-contact-monopole}. We construct examples using
Boothby--Wang fibrations and prove that, in this setting, the equation reduces
to the primitivity condition on the curvature over the Kähler base. We also
relate contact monopoles to anti-self-dual contact instantons. We then prove a
compact rigidity result for contact metric $5$-manifolds: every smooth contact
monopole has parallel Higgs field and satisfies the reduced curvature equation
$F_A\wedge\omega=0$. In the Sasakian case, this connects compact contact
monopoles with anti-self-dual contact instantons under the usual transverse
type condition. Finally, we describe the deformation complex associated with
the equation and show that it is not elliptic in the usual sense.

In Section \ref{sec:6}, we study the contact monopole equation on unit tangent
sphere bundles. After recalling the canonical contact metric structure on
$T_1N$, we derive the local form of the equation in an adapted frame. As an
application, we prove that pure pullback configurations from the base are
rigid: they solve the contact monopole equation only in the flat and parallel
case.

In Section \ref{sec:7}, we study the homogeneous Sasakian manifold
$T_1S^3\cong SO(4)/SO(2)$. We first describe a reductive model compatible with
the canonical contact metric structure. Then we review homogeneous principal
bundles and invariant connections via Wang's theorem.

For homogeneous $U(1)$-bundles, we classify invariant contact monopoles and
obtain non-flat solutions for nonzero topological charge. We then consider the
first non-abelian case, namely homogeneous $SO(3)$-bundles. In this case, we
classify the invariant contact monopoles in terms of the weight $m$ of the
isotropy homomorphism: the cases $m\neq1$ reduce to the maximal torus, while
the case $m=1$ produces a genuinely non-abelian branch.

%===================================================================== 
\section{Contact monopoles: Basic examples and rigidity}\label{sec:2} 
%===================================================================== 
Let $(M,\eta,\xi,g,\Phi)$ be a contact metric $5$-manifold, let $P\to M$ be a principal $G$-bundle, where $G$ is a compact Lie group, and let $\mathfrak g_P$ be the adjoint bundle. In the contact setting, the natural distinguished $2$-form is the transverse fundamental form $\omega:=d\eta.$ 
\begin{definition} 
Let $A\in \mathcal A(P)$ be a connection and let $\phi\in \Omega^0(M,\mathfrak g_P)$ be a section. A pair $(A,\phi)$ satisfying \begin{equation}\label{eq:monopole 5} 
\ast(F_A\wedge \omega)=\nabla_A\phi 
\end{equation} 
is called a \textit{contact monopole}. The section $\phi$ is called a \textit{Higgs field}. 
\end{definition} 
We use the notations $d_A$ and $\nabla_A$ interchangeably when acting on sections of the adjoint bundle. 
%--------------------------------------------------------------- 
\subsection{Boothby--Wang lifts and first example}\label{sec:21} 
%--------------------------------------------------------------- 
Let $(X,\omega_X)$ be a compact Hodge surface, i.e., a compact Kähler surface with integral Kähler class $[\omega_X]\in H^2(X,\mathbb Z)$. By the Boothby--Wang construction \cite{boyer2007sasakian}, there exists a principal $S^1$-bundle $ \pi\colon M \to X$ with connection $1$-form $\eta$ satisfying $d\eta=\pi^\ast \omega_X.$ The total space $M$ carries a natural regular Sasakian structure. 
\begin{example}\label{ex:CP1CP1} 
Let $M\xrightarrow{\pi}  \CP^1\times\CP^1$ be the Boothby--Wang fibration associated with the integral Kähler form $\omega=\omega_1+\omega_2,$ where $\omega_i$ denotes the Fubini--Study form on the $i$-th factor. 
Consider the holomorphic line bundle $L=\mathcal O(k,-k)\longrightarrow \CP^1\times\CP^1,$ equipped with its natural Hermitian connection $A_L$ \cite{GriffithsHarris1978,Hartshorne1977}. Its curvature is given by 
$$ 
F_{A_L}=ik(\omega_1-\omega_2). 
$$ 
Let $A:=\pi^\ast A_L$ be the pullback connection on $\pi^\ast L\to M$. Since $F_A=\pi^\ast F_{A_L}$ and, by the Boothby--Wang construction, $d\eta=\pi^\ast\omega,$ we obtain $$ F_A\wedge d\eta=\pi^\ast(F_{A_L}\wedge\omega). $$ Now, 
$$ 
F_{A_L}\wedge\omega = ik(\omega_1-\omega_2)\wedge(\omega_1+\omega_2)=0, 
$$ 
and consequently $F_A\wedge d\eta=0.$  Since the structure group is abelian, the adjoint representation is trivial. Hence the adjoint bundle is canonically trivial, and any constant section $\phi\in\Omega^0(M,i\mathbb R)$ is $A$-parallel. Therefore 
$$ 
\ast(F_A\wedge d\eta)=0=\nabla_A\phi, 
$$ 
and the pair $(A,\phi)$ defines a contact monopole \eqref{eq:monopole 5} on the Sasakian manifold $M^5$. 
\end{example} 

 This  motivates the following characterization result.
\begin{proposition}\label{prop:BWlift} 
Let $S^1\to M\to X^4$ be a Boothby--Wang fibration over a compact Kähler surface $(X,\omega_X)$  \cite{BoothbyWang1958}. Let $P_X\to X$ be a principal $G$-bundle endowed with a connection $A_X$, and let $P\to M$ be the pullback bundle  $P:=\pi^\ast P_X$ 
$$
\begin{CD}  P @>{ }>> P_X\\  @V{ }VV @VV{ }V\\  M @>{\pi}>> X. 
\end{CD}
$$ 
Assume that $\phi\in\Omega^0(M,\mathfrak g_P)$ is $A$-parallel, where   $A=\pi^\ast A_X$. Then the pair $(A,\phi)$ satisfies the contact monopole equation \eqref{eq:monopole 5} if and only if $F_{A_X}\wedge\omega_X=0.$ 
In particular, if $A_X$ is ASD, then every $A$-parallel Higgs field defines a contact monopole on $M$. 
\end{proposition} 
\begin{proof} 
Since $A=\pi^\ast A_X$, its curvature is $F_A=\pi^\ast F_{A_X}.$
Moreover, by the Boothby--Wang construction, $d\eta=\pi^\ast\omega_X.$ Therefore
$$ 
F_A\wedge d\eta = \pi^\ast F_{A_X}\wedge \pi^\ast\omega_X = \pi^\ast(F_{A_X}\wedge\omega_X). 
$$ 
Since $\phi$ is $A$-parallel, we have $\nabla_A\phi=0.$ Thus, for the present pair $(A,\phi)$, the contact monopole equation \eqref{eq:monopole 5} reduces to 
$$
F_A\wedge d\eta=0, 
$$
since the Hodge star is an isomorphism. Using the identity above, we obtain
$F_A\wedge d\eta=0$ if and only if $\pi^\ast(F_{A_X}\wedge\omega_X)=0.$ Since $\pi:M\to X$ is a surjective submersion, the pullback map $\pi^\ast$ is injective on differential forms. Hence
$\pi^\ast(F_{A_X}\wedge\omega_X)=0$ if and only if $F_{A_X}\wedge\omega_X=0.$ Therefore $(A,\phi)$ satisfies the contact monopole equation \eqref{eq:monopole 5} if and only if
$$
F_{A_X}\wedge\omega_X=0.
$$
Finally, assume that $A_X$ is anti-self-dual on the Kähler surface $X$. With
respect to the Kähler orientation, the Kähler form $\omega_X$ is self-dual.
Hence an anti-self-dual $2$-form is pointwise orthogonal to $\omega_X$, and so
$$
F_{A_X}\wedge\omega_X=0.
$$
Thus every $A$-parallel Higgs field defines a contact monopole on $M$.
\end{proof} 
The Boothby--Wang construction shows that contact monopoles are naturally related to primitive curvature on Kähler surfaces. We now turn to the intrinsic contact-geometric interpretation of this condition in terms of contact instantons. 
%-------------------------------------------------------------------------- 
\subsection{Contact instantons and compact monopoles}\label{sec:3} 
%--------------------------------------------------------------------------- 
We use the following definition from \cite{baraglia2016moduli}.
\begin{definition} 
Let $(M^5,\eta,\xi,\Phi,g)$ be a Sasakian manifold. A connection $A$ is called a \textbf{self-dual contact instanton} (SDCI) for $\lambda=1$ and an \textbf{anti-self-dual contact instanton} (ASDCI) for $\lambda=-1$ if 
\begin{equation}\label{eq:SDCI} 
\ast(F_A\wedge\eta)=\lambda F_A. 
\end{equation} 
\end{definition} 
The space of $k$-forms on $M$ decomposes as 
\begin{equation}\label{eq:kformsdecomposition}
\Omega^k(M)=\Omega_H^k(M)\oplus \eta\wedge\Omega_H^{k-1}(M), 
\end{equation} 
where $\Omega_H^k(M)$ denotes horizontal $k$--forms with respect to the Reeb foliation, i.e.,   $\iota_\xi\alpha=0$ \cite{boyer2007sasakian}. The horizontal $2$-forms further decompose into self-dual and anti-self-dual parts \cite{baraglia2016moduli}: 
$$ 
\Omega_H^2(M)=\Omega_H^{2,+}(M)\oplus\Omega_H^{2,-}(M), 
$$ 
where $\alpha\in\Omega_H^{2,\pm}$ if 
$$ 
\ast(\alpha\wedge\eta)=\pm \alpha. 
$$ 
This decomposition is compatible with the notion of contact instantons in \eqref{eq:SDCI}. The horizontal distribution carries a transverse Kähler structure. Hence the complexified bundle of transverse $2$-forms decomposes as 
$$ 
\Omega^2(M,\mathfrak g_P)= \Omega_H^{2,0}(M,\mathfrak g_P) \oplus \Omega_H^{0,2}(M,\mathfrak g_P) \oplus \Omega_{H,0}^{1,1}(M,\mathfrak g_P) \oplus \langle\omega\rangle\otimes\Omega^0(M,\mathfrak g_P). 
$$ 
Here 
$$ 
\Omega_{H,0}^{1,1}(M,\mathfrak g_P) = \{\alpha\in\Omega_H^{1,1}(M,\mathfrak g_P)\mid \Lambda\alpha=0\} 
$$ 
denotes the primitive $(1,1)$-component, i.e., the $(1,1)$-forms orthogonal to $\omega$. Accordingly, the curvature decomposes as 
$$ 
F_A= F_A^{2,0} + F_A^{0,2} + F_{A,0}^{1,1} + f\,\omega + \eta\wedge\beta, $$ 
where $F_{A,0}^{1,1}$ is primitive, $f\in\Omega^0(M,\mathfrak g_P)$, and $\beta\in\Omega_H^1(M,\mathfrak g_P)$. We recall that on a Sasakian $5$-manifold, the curvature of an ASDCI is transverse of type $(1,1)$ and primitive \cite{baraglia2016moduli}.
\begin{lemma}\label{lem:ASDcovariantThenMonopole} 
Let $(M,\eta,\xi,\Phi,g)$ be a Sasakian $5$-manifold, and let
$(A,\phi)$ be a smooth configuration. If $A$ is an ASDCI and $\phi$ is
$A$-parallel, then $(A,\phi)$ is a solution of the contact monopole equation.
\end{lemma} 
\begin{proof}
If $A$ is an ASDCI, then its curvature is horizontal and primitive, i.e.,  $F_A\wedge \omega=0,$ where $\omega=d\eta$ is the fundamental $2$-form \cite{baraglia2016moduli}. Since $\phi$ is $A$-parallel, it follows that
$$
\ast(F_A\wedge\omega)=0=\nabla_A\phi,
$$
and hence $(A,\phi)$ satisfies the contact monopole equation
\eqref{eq:monopole 5}.
\end{proof}
We now record a general compactness phenomenon for the contact monopole equation. On compact contact metric $5$-manifolds, the equation forces the Higgs field to be parallel. Consequently, the contact monopole equation decouples into the reduced curvature condition $F_A\wedge\omega=0$ and the parallel Higgs condition $\nabla_A\phi=0$.
\begin{lemma}\label{lem:compact-contact-monopoles}
Let $(M,\eta,\xi,\Phi,g)$ be a compact contact metric $5$-manifold, and let
$(A,\phi)$ be a smooth contact monopole. Then $\phi$ is $A$-parallel. In
particular, $F_A\wedge\omega=0.$ Moreover, if $A$ is irreducible, then $\phi=0$.
\end{lemma}
\begin{proof}
Since $(A,\phi)$ satisfies the contact monopole equation, we have
\[
F_A\wedge\omega=\ast d_A\phi.
\]
Applying $d_A$ and using the Bianchi identity $d_AF_A=0$ together with
$d\omega=0$, we obtain
\[
0=d_A(F_A\wedge\omega)=d_A(\ast d_A\phi).
\]
Thus $d_A^\ast d_A\phi=0.$ Taking the $L^2$ inner product with $\phi$ gives
\[
0=\langle d_A^\ast d_A\phi,\phi\rangle_{L^2} = \|d_A\phi\|^2_{L^2}.
\]
Hence $d_A\phi=0$. Substituting this into the contact monopole equation gives
$F_A\wedge\omega=0.$ Finally, if $A$ is irreducible, every $A$-parallel adjoint-valued section  vanishes. Therefore $\phi=0$.
\end{proof}
\begin{corollary}\label{cor:compact-sasakian-monopoles-ASDCI}
Let $(M,\eta,\xi,\Phi,g)$ be a compact Sasakian $5$-manifold, and let
$(A,\phi)$ be a smooth contact monopole. Then $\phi$ is $A$-parallel and
$F_A\wedge\omega=0.$ If, in addition, $F_A^{2,0}=F_A^{0,2}=0$, then $A$ is an ASDCI.  
\end{corollary}
\begin{proof}
By Lemma \ref{lem:compact-contact-monopoles}, every smooth contact monopole
on a compact contact metric $5$-manifold has $A$-parallel Higgs field and
satisfies $F_A\wedge\omega=0.$ In particular, these conclusions hold on compact Sasakian $5$-manifolds.

Assume now that $F_A^{2,0}=F_A^{0,2}=0$. Then $F_A$ is of transverse
type $(1,1)$  and primitive. By the  characterization of anti-self-dual contact instantons on Sasakian $5$-manifolds  \cite{baraglia2016moduli}, this is equivalent to $A$ being an ASDCI.  
\end{proof} 
\begin{example}[A non-example on the Hopf fibration] Consider the standard Hopf fibration  $S^5 \xrightarrow{\pi} \CP^2,$  where $\CP^2$ is endowed with the Fubini--Study Kähler form $\omega_{FS}$. The standard Sasakian structure on $S^5$ satisfies 
$$ 
d\eta=\pi^\ast \omega_{FS}. 
$$ 
Let $\mathcal O(-1)\to \CP^2$ denote the tautological line bundle equipped with its Chern connection $A$ \cite{GriffithsHarris1978,Hartshorne1977}. Its curvature is given by 
$
F_A=i\,\omega_{FS}.
$ 
In particular, $F_A$ is of type $(1,1)$. However, 
$$ 
F_A\wedge\omega_{FS}= i\,\omega_{FS}^2 \neq 0. 
$$ 
Hence $F_A$ is not primitive. Let $\widetilde A=\pi^\ast A$ be the pullback connection to $S^5$. Then 
$$ 
F_{\widetilde A}\wedge d\eta = \pi^\ast (F_A\wedge\omega_{FS}) = i\,\pi^\ast (\omega_{FS}^2) \neq 0. 
$$ 
Consequently, $\ast(F_{\widetilde A}\wedge d\eta)\neq 0.$ If $(\widetilde A,\phi)$ were a smooth contact monopole on $S^5$, then, since
$S^5$ is compact, Lemma \ref{lem:compact-contact-monopoles}
would imply that $\phi$ is $\widetilde A$-parallel. Hence
$$
\nabla_{\widetilde A}\phi=0,
$$
which contradicts $\ast(F_{\widetilde A}\wedge d\eta)\neq 0.$ Therefore, the fixed pullback connection $\widetilde A$ cannot occur in any  smooth contact monopole pair $(\widetilde A,\phi)$ on $S^5$.
\end{example}
This example shows that being of type $(1,1)$ is not sufficient for a lifted
connection to define a contact monopole. The primitive condition
$F_A\wedge\omega_{FS}=0$ is essential.
%--------------------------------------------------------------- 
\subsection{The deformation complex and its degeneracy}\label{sec:4} 
%--------------------------------------------------------------- 
Let $(A,\phi)$ be a smooth configuration, which need not satisfy the monopole equation \eqref{eq:monopole 5}. Consider the nonlinear map 
$$ 
\mathcal{F}(A,\phi):=\ast(F_A\wedge\omega)-\nabla_A\phi \in \Omega^1(M,\mathfrak g_P). 
$$ 
Then $(A,\phi)$ is a monopole if and only if $\mathcal F(A,\phi)=0$. The linearization of the monopole equation at $(A,\phi)$ is 
$$
\begin{aligned} 
d_2:(\varphi,a)\in \Omega^0(M,\mathfrak g_P)\oplus\Omega^1(M,\mathfrak g_P) &\mapsto \ast (d_A a\wedge\omega)-\bigl(\nabla_A\varphi+[a,\phi]\bigr)\in  \Omega^1(M,\mathfrak g_P)  
\end{aligned} 
$$
Moreover, the infinitesimal action of the gauge group is 
$$
d_1:s\in \Omega^0(M,\mathfrak g_P) \mapsto ([s,\phi],-\nabla_As)\in  \Omega^0(M,\mathfrak g_P)\oplus\Omega^1(M,\mathfrak g_P).
$$
We write $\nabla_A=d_A$ on adjoint-valued functions. These maps fit into the sequence 
\begin{equation}\label{eq:MonopoleComplex} 
0\longrightarrow \Omega^0(M,\mathfrak g_P) \xrightarrow{d_1} \Omega^0(M,\mathfrak g_P)\oplus\Omega^1(M,\mathfrak g_P) \xrightarrow{d_2} \Omega^1(M,\mathfrak g_P) \longrightarrow 0. 
\end{equation} 
\begin{lemma} 
If $(A,\phi)$ is a monopole, then the sequence \eqref{eq:MonopoleComplex} is a complex. 
\end{lemma} 
\begin{proof} 
For $s\in\Omega^0(M,\mathfrak g_P)$, we compute 
\begin{align*} 
d_2(d_1s) &= d_2([s,\phi],-\nabla_As)
= \ast\bigl(d_A(-\nabla_As)\wedge\omega\bigr) 
-\bigl(\nabla_A[s,\phi]+[-\nabla_As,\phi]\bigr)\\ 
&= -\ast([F_A,s]\wedge\omega)-[s,\nabla_A\phi]
= [s,\ast(F_A\wedge\omega)]-[s,\nabla_A\phi]\\ 
&= [s,\ast(F_A\wedge\omega)-\nabla_A\phi]. 
\end{align*} 
Hence, if $(A,\phi)$ satisfies \eqref{eq:monopole 5}, then $d_2\circ d_1=0$. \end{proof} 
\begin{remark} 
The deformation complex \eqref{eq:MonopoleComplex}  is not elliptic in the usual sense. Indeed, its principal symbol at a covector $\zeta\in T_x^\ast M\setminus\{0\}$ is given by 
$$ 
\sigma_\zeta(d_1)(s)=(0,-\zeta s), \qquad \sigma_\zeta(d_2)(\varphi,a) = \ast(\zeta\wedge a\wedge\omega)-\zeta\,\varphi. 
$$ 
If $\zeta$ is horizontal, a direct computation in a local coframe adapted to the contact structure shows that the image of $\sigma_\zeta(d_2)$ is contained in a proper subspace of $T_x^\ast M\otimes\mathfrak g_P$. In particular, $\sigma_\zeta(d_2)$ is not surjective, and hence the sequence of symbols is not exact.  
\end{remark} 
The failure of ellipticity reflects the fact that the operator
$\alpha\mapsto \ast(\alpha\wedge\omega)$ does not detect all curvature components, in particular the primitive  $(1,1)$-component of a horizontal $2$-form. Hence, in contrast with the $G_2$-monopole setting \cite{oliveira2014monopoles}, the contact monopole equation does not directly provide a standard elliptic deformation theory for smooth monopoles. For this reason, we do not pursue a general smooth moduli
space theory here. Instead, we focus on geometric constructions and explicit
reductions, first on unit tangent sphere bundles and then on the homogeneous
model $T_1S^3$. 
%----------------------------------------------------------------
\section{Contact monopoles on unit tangent sphere bundles}\label{sec:6}
%-------------------------------------------------------------
Throughout this section, $(N^3,g_N)$ denotes an oriented Riemannian
$3$-manifold. Its unit tangent sphere bundle $T_1N$ is the space of all unit
tangent vectors in $TN$.   The canonical projection will be denoted by
\begin{equation}\label{eq:canonicalprojection}
\pi:T_1N\longrightarrow N.
\end{equation}
The canonical contact form is induced by the tautological $1$-form,
$\eta_{(x,u)}(Z)=g_N(u,d\pi(Z))$. 
%--------------------------------------------------------------------
\subsection{The canonical contact metric structure on unit tangent  bundles}\label{sec:5}
%-------------------------------------------------------------------
The unit tangent sphere  bundle $T_1N$ \eqref{eq:canonicalprojection}  is a $5$-dimensional manifold carrying a canonical contact structure   \cite{Tashiro1969,Blair2002,CalvarusoUnitTangent}.
The Levi--Civita connection of $(N^3,g_N)$ determines a splitting of the tangent bundle of $TN$ into horizontal and vertical subbundles
$$
TTN = H \oplus V.
$$
If $X$ is a vector field on $N$, we denote by $X^h$ and $X^v$ its horizontal
and vertical lifts, respectively. At  $(x,u)\in T_1N$, the tangent space
to the fiber  sphere consists of those vertical vectors orthogonal to $u$. Hence  the vertical lift $X^v$ is tangent to $T_1N$ only after removing its component in the radial direction $u$. This leads to the tangential lift
$$
X^t_{(x,u)}
=
\bigl(X_x-g_N(X_x,u)u\bigr)^v.
$$
Thus, at a point $(x,u)\in T_1N$, the tangent space of $T_1N$ is spanned by
horizontal lifts and tangential vertical lifts. The canonical contact form on $T_1N$ is given by
\begin{equation}\label{eq:canonical-contact-form}
\eta_{(x,u)}(Z)=g_N(u,d\pi(Z)), \;\text{ for $Z\in T_{(x,u)}T_1N$.}
\end{equation}
Equivalently, $\eta$ is the restriction to $T_1N$ of the tautological $1$-form $\Theta\in\Omega^1(TN)$ defined by
\begin{equation}\label{eq:tautological-form}
\Theta_{(x,u)}(Z)=g_N(u,d\pi(Z)).
\end{equation}
Its Reeb vector field is the vector field generating the geodesic flow on
$T_1N$, namely the horizontal lift of the unit vector $u$:
\begin{equation}\label{eq:reeb-unit-tangent}
\xi_{(x,u)}=u^h.
\end{equation}
Therefore, the Reeb flow on $T_1N$ is precisely the geodesic flow of $(N,g_N)$. We now recall the metric and tensor fields entering the standard contact metric structure on $T_1N$ \cite{Tashiro1969,Blair2002,CalvarusoUnitTangent}. For vector fields $X,Y$ on $N$, the Sasaki metric\footnote{Here ``Sasaki metric'' refers to the natural Riemannian metric on the tangent bundle $TN$ induced by $g_N$ and the Levi--Civita connection. It should not be confused with a Sasakian metric: the manifold $TN$ has even dimension, whereas Sasakian manifolds are odd-dimensional.} on $TN$ is defined by
$$
g_S(X^h,Y^h)=g_N(X,Y)=g_S(X^v,Y^v),\qquad g_S(X^h,Y^v)=0.
$$ 
By restriction to the unit tangent bundle $T_1N$, and up to the
conventional homothetic normalization, one obtains a Riemannian metric on
$T_1N$, denoted by $g$. The \textit{Sasaki metric} on $TN$ carries a natural almost complex structure $J:TTN\to TTN$ defined by
$$
J(X^h)=X^v,
\qquad
J(X^v)=-X^h.
$$
The induced contact metric structure on the hypersurface $T_1N\subset TN$
is obtained by projecting $J$ tangentially onto $TT_1N$. More precisely,
if $X\in T_xN$ satisfies $g_N(X,u)=0$, then
$$
\Phi(X^h)=X^t,
\qquad
\Phi(X^t)=-X^h,
\qquad
\Phi(\xi)=0.
$$
Then 
$(\eta,\xi,\Phi,g)$ defines the standard contact metric structure on $T_1N$ \cite{Blair2002,CalvarusoUnitTangent}.

The following identities are standard for the canonical contact structure on
unit tangent sphere bundles; see, for example, \cite{Tashiro1969,Blair2002,CalvarusoUnitTangent}. We include a brief proof in order to fix the sign conventions used later. 
\begin{lemma}\label{lem:deta-unit-tangent}
Let $(N^3,g_N)$ be an oriented Riemannian $3$-manifold and let
$T_1N$ be its unit tangent sphere bundle \eqref{eq:canonicalprojection}, endowed with the canonical contact form $\eta$ \eqref{eq:canonical-contact-form}. For vector fields $X,Y$ on $N$, the differential of $\eta$ satisfies
$$
d\eta(X^h,Y^h)=0, \qquad d\eta(X^t,Y^t)=0,
\quad\text{and} \quad d\eta(X^h,Y^t) = -g_N(X,Y)+g_N(X,u)g_N(Y,u).
$$
In particular, if $X,Y\perp u$, then $d\eta(X^h,Y^t)=-g_N(X,Y).$
\end{lemma}
\begin{proof}
The identities are standard for the canonical contact form on $T_1N$. We give  the short argument for completeness and to fix signs. Recall that $\eta$ is the restriction to $T_1N$ of the tautological $1$-form $\Theta$ on $TN$ \eqref{eq:tautological-form}. Thus, for a horizontal lift $X^h$ and a vertical lift $Y^v$, one has
$$
\Theta(X^h)=g_N(u,X),
\qquad
\Theta(Y^v)=0.
$$
Using Cartan's formula and the standard identity
$[X^h,Y^v]=(\nabla_XY)^v$, we obtain
$$
d\Theta(X^h,Y^v)=-g_N(X,Y).
$$
Since $Y^v$ is not necessarily tangent to $T_1N$, we replace it by the
tangential lift $Y^t=(Y-g_N(Y,u)u)^v$. Therefore
\begin{align*}
d\eta(X^h,Y^t)
&= d\Theta\bigl(X^h,(Y-g_N(Y,u)u)^v\bigr) = -g_N\bigl(X,Y-g_N(Y,u)u\bigr)\\
&= -g_N(X,Y)+g_N(X,u)g_N(Y,u).
\end{align*}
The pure horizontal and pure tangential vertical terms vanish by the same
standard structure equations for the horizontal--vertical splitting. Hence
$$
d\eta(X^h,Y^h)=0, \qquad d\eta(X^t,Y^t)=0.
$$
This proves the lemma.
\end{proof}  
%-------------------------------------------------------------
\subsection{The local contact monopole system on $T_1N$}
%-------------------------------------------------------------
Let $(x,u)\in T_1N$ and choose an orthonormal basis $\{e_1,e_2,u\}$ of $T_xN$, where $e_1,e_2\perp u$. We define the adapted frame
\begin{equation}\label{eq:adaptedFrame}
E_1=e_1^h,\; E_2=e_2^h,\; E_3=e_1^t,\; E_4=e_2^t,\; E_5=\xi=u^h.    
\end{equation} 
Then $\{E_1,E_2,E_3,E_4,E_5\}$ is a local frame on $T_1N$ adapted to the contact metric structure. We denote by $\{E^1,E^2,E^3,E^4,\eta\}$ the dual coframe, characterized by
\begin{equation}\label{eq:coframe}
E^i(E_j)=\delta^i_j,\qquad \eta(E_5)=1, \qquad \eta(E_i)=0,\quad i=1,\cdots,4.    
\end{equation}
It follows from Lemma \ref{lem:deta-unit-tangent} that, with our convention, the fundamental $2$-form is given by \begin{equation}\label{eq:detaadapted} 
d\eta=-(E^1\wedge E^3+E^2\wedge E^4). 
\end{equation} 
Consequently,  $\frac12 d\eta\wedge d\eta=-E^{1234}.$ 
We choose the orientation determined by the ordered coframe 
$\{E^1,E^2,E^3,E^4,\eta\}$. Thus the Riemannian volume form is \begin{equation}\label{eq:volume-adapted-frame} 
\vol_g=E^{1234}\wedge\eta = -\frac12 d\eta\wedge d\eta\wedge\eta. 
\end{equation}
We now compute the contact monopole equation in the adapted frame
\eqref{eq:adaptedFrame}. Let $A$ be a connection on a principal bundle over
$T_1N$, with curvature $F_A\in\Omega^2(T_1N,\mathfrak g_P).$
According to  \eqref{eq:kformsdecomposition}, the curvature $F_A$ decomposes as $F_A=F_H+F_\eta$, where 
\begin{equation}\label{eq:FH}
F_H = F_{12}E^{12} +F_{13}E^{13} +F_{14}E^{14} +F_{23}E^{23} +F_{24}E^{24}
+F_{34}E^{34}.
\end{equation}
and 
\begin{equation}\label{eq:Feta}
F_\eta = \left(F_{15}E^1 +F_{25}E^2 +F_{35}E^3 +F_{45}E^4\right)\wedge\eta,
\end{equation}
where $E^{ij}=E^i\wedge E^j,$ and  $ F_{i5}\in\Omega^0(T_1N,\mathfrak g_P).$  Using \eqref{eq:detaadapted}, we compute
\begin{align*}
F_A\wedge d\eta
&= -F_A\wedge E^{13}-F_A\wedge E^{24} \\
&= (F_{24}+F_{13})E^{1234} + (F_{25}E^{123} -F_{45}E^{134} -F_{15}E^{124} +F_{35}E^{234})\wedge\eta.
\end{align*}
From the above computation, the left-hand side of \eqref{eq:monopole 5} is
\begin{equation}\label{eq:astF_Aomega}
\begin{aligned}
\ast(F_A\wedge d\eta) &= (F_{24}+F_{13})\eta -F_{25}E^4 +F_{45}E^2 -F_{15}E^3 +F_{35}E^1  \\
&= F_{35}E^1 +F_{45}E^2 -F_{15}E^3 -F_{25}E^4    +(F_{24}+F_{13})\eta.
\end{aligned}    
\end{equation}
On the other hand, with respect to the adapted coframe \eqref{eq:coframe}, the covariant derivative of the Higgs field decomposes as
\begin{equation}\label{eq:cov-derivative-higgs-adapted}
\nabla_A\phi = (\nabla_{A,E_1}\phi)E^1 + (\nabla_{A,E_2}\phi)E^2
+ (\nabla_{A,E_3}\phi)E^3 + (\nabla_{A,E_4}\phi)E^4 + (\nabla_{A,\xi}\phi)\eta.
\end{equation}
The equality $\ast(F_A\wedge d\eta)=\nabla_A\phi$ in \eqref{eq:monopole 5} holds if and only if the corresponding coefficients of \eqref{eq:astF_Aomega} and \eqref{eq:cov-derivative-higgs-adapted} agree. We summarize this in the following local characterization result.
\begin{proposition}\label{prop:local-monopole-system}
Let $T_1N$ be the unit tangent sphere bundle of an oriented Riemannian
$3$-manifold $(N^3,g_N)$, with canonical projection
$\pi:T_1N\to N$ as in \eqref{eq:canonicalprojection}, endowed with its canonical contact metric structure \eqref{eq:canonical-contact-form}--\eqref{eq:reeb-unit-tangent}. Let $A$ be a connection  on a principal bundle over $T_1N$ and let $F_A=F_H+F_\eta$ be the decomposition \eqref{eq:FH}--\eqref{eq:Feta}. Then, with respect to the adapted frame \eqref{eq:adaptedFrame},  the contact monopole equation \eqref{eq:monopole 5} is equivalent to the system
\begin{equation}\label{eq:local-monopole-system}
\begin{aligned}
\nabla_{A,E_1}\phi &= F_{35}, & \nabla_{A,E_2}\phi &= F_{45},&\nabla_{A,\xi}\phi &= F_{24}+F_{13}.\\
\nabla_{A,E_3}\phi &= -F_{15},& \nabla_{A,E_4}\phi &= -F_{25}.
\end{aligned}
\end{equation}
\end{proposition}
\begin{corollary}\label{cor:parallel-higgs-components}
Under the assumptions of Proposition \ref{prop:local-monopole-system},
if the Higgs field $\phi$ is $A$-parallel, then the contact monopole equation
is equivalent to
$$
F_{15}=F_{25}=F_{35}=F_{45}=F_{13}+F_{24}=0.
$$ 
\end{corollary} 
%---------------------------------------------------------------------
\subsection{A rigidity result for pure pullbacks}
\label{subsec:pure-pullback-rigidity}
%---------------------------------------------------------------------
We record a rigidity result showing that pure pullback configurations from the base do not produce nontrivial contact monopoles.
\begin{proposition}\label{prop:pure-pullback-rigidity}
Let $T_1N$ be the unit tangent bundle \eqref{eq:canonicalprojection} over an oriented Riemannian $3$-manifold $N$. Let $P_N\to N$ be a principal $G$-bundle with connection $a$, and let $\Phi\in\Omega^0(N,\mathfrak g_{P_N})$. Consider the pullback bundle $P:=\pi^\ast P_N\to T_1N$, with pullback connection $A:=\pi^\ast a$ and pullback Higgs field $\phi:=\pi^\ast\Phi$. Then $(A,\phi)$ satisfies the contact monopole equation \eqref{eq:monopole 5} on $T_1N$ if and only if $ F_a=0$ and  $\nabla_a\Phi=0. $
\end{proposition}
\begin{proof}
Assume first that $(A,\phi)$ satisfies the contact monopole equation
\eqref{eq:monopole 5}. Let $(x,u)\in T_1N$ and choose an orthonormal basis
$\{e_1,e_2,u\}$ of $T_xN$. In the adapted frame \eqref{eq:adaptedFrame}, we have
$$
d\pi(E_1)=e_1,\quad
d\pi(E_2)=e_2,\quad
d\pi(E_3)=0,\quad
d\pi(E_4)=0,\quad
d\pi(E_5)=u.
$$
Since $A=\pi^\ast a$, the curvature satisfies   $F_A(E_i,E_j)=F_a(d\pi(E_i),d\pi(E_j))$. Therefore
$$
F_{35}=F_{45}=F_{13}=F_{24}=0,
$$
while $F_{15}=F_a(e_1,u)$ and  $F_{25}=F_a(e_2,u).$ 
Similarly, since $\phi=\pi^\ast\Phi$, we have 
$$
(\nabla_A\phi)_{(x,u)}(E_i)=(\nabla_a\Phi)_x\bigl(d\pi_{(x,u)}(E_i)\bigr).
$$
Hence
$$
\nabla_{A,E_1}\phi=\nabla_{a,e_1}\Phi, \qquad \nabla_{A,E_2}\phi=\nabla_{a,e_2}\Phi, \qquad
\nabla_{A,E_5}\phi=\nabla_{a,u}\Phi,
$$
and $\nabla_{A,E_3}\phi=\nabla_{A,E_4}\phi=0.$ Using the system \eqref{eq:local-monopole-system} in Proposition \ref{prop:local-monopole-system}, the contact monopole equation is equivalent to
\begin{equation}\label{eq:nablaNegativeresult}
\nabla_{a,e_1}\Phi=0,\qquad
\nabla_{a,e_2}\Phi=0,\qquad
\nabla_{a,u}\Phi=0,
\end{equation}
and $F_a(e_1,u)=0$ and $F_a(e_2,u)=0.$ Since the contact monopole equation holds at every point $(x,u)\in T_1N$, the previous identities hold for every unit vector $u\in T_xN$. Thus, for  each such $u$ and for all 
$v\perp u$  
$$
F_a(v,u)=0.
$$
Now let $v\in T_xN$ be arbitrary. We decompose
$v=v^\perp+\lambda u$, with $v^\perp\perp u$. Since $F_a$ is a $2$-form,
$$
F_a(v,u)=F_a(v^\perp,u)+\lambda F_a(u,u)=0.
$$
Hence $F_a(v,u)=0$ for all $v\in T_xN$ and all unit vectors $u\in T_xN$.
Since $u$ is arbitrary, it follows that $F_a=0.$ Similarly, the identities \eqref{eq:nablaNegativeresult} hold for an orthonormal basis 
$\{e_1,e_2,u\}$ of $T_xN$, and therefore
$\nabla_a\Phi=0.$

Conversely, suppose that $F_a=0$ and $\nabla_a\Phi=0$. Then
$$
F_A=\pi^\ast F_a=0, \qquad \nabla_A\phi=\pi^\ast(\nabla_a\Phi)=0.
$$
Therefore $\ast(F_A\wedge d\eta)=0=\nabla_A\phi,$ so $(A,\phi)$ satisfies the contact monopole equation.
\end{proof}
Proposition \ref{prop:pure-pullback-rigidity} shows that the contact monopole equation on $T_1N$ cannot be obtained from the classical Bogomolny equation on $N$ by a naive pullback construction. Indeed, pure pullbacks solve the contact monopole equation only in the flat and parallel case. Thus the equation defines a genuinely $5$-dimensional gauge-theoretic problem on $T_1N$. 
%---------------------------------------------------------------------
\section{Homogeneous contact monopoles on $T_1S^3$}\label{sec:7}
%---------------------------------------------------------------------
We now specialize to the round sphere $N=S^3$. In this case, 
\begin{equation}\label{eq:T1S3stiefel}
T_1S^3\cong V_2(\mathbb R^4)\cong SO(4)/SO(2).
\end{equation} 
The canonical contact metric structure is homogeneous \cite{CalvarusoPerrone2010}. This homogeneous description allows us to study invariant contact monopoles by means of homogeneous principal bundles and Wang's theorem.  We first fix a reductive model for $V_2(\mathbb R^4)$ compatible with the canonical contact structure. We then recall Wang's theorem and reduce the contact monopole equation to explicit finite-dimensional algebraic systems. 
%---------------------------------------------------------------------
\subsection{A reductive model for the homogeneous contact structure on $T_1S^3$}\label{sec:homogeneous-reductive-model}
%---------------------------------------------------------------------
By Tashiro's theorem \cite{Tashiro1969}, the unit tangent sphere bundle
$T_1N$ is K-contact if and only if $(N,g_N)$ has constant sectional curvature $1$; in this case, $T_1N$ is Sasakian
\cite{Tashiro1969,Blair2002,CalvarusoUnitTangent}. Hence, for the round
sphere $S^3$, the unit tangent sphere bundle $T_1S^3$ \eqref{eq:T1S3stiefel} is a Sasakian $5$-manifold.  

We now fix a reductive model for this homogeneous space, compatible with the canonical contact metric structure.  Let $E_{ij}$ denote the elementary $4\times4$ matrix with a $1$ in the $(i,j)$-entry and zeros elsewhere. We fix the following basis of the Lie algebra $\so(4)$:
\begin{align}\label{eq:so4-basis}
\begin{split}
e_1&=E_{23}-E_{32},\qquad e_2=E_{24}-E_{42},\qquad e_3=E_{13}-E_{31},\\[0.3em]
e_4&=E_{14}-E_{41},\qquad e_5=E_{12}-E_{21},\qquad e_6=E_{34}-E_{43}.
\end{split}
\end{align}
The isotropy algebra is $\fh=\Span\{e_6\}\cong\so(2),$ and we take the reductive complement $\so(4)=\fh\oplus\fm,$ where 
$\fm=\Span\{e_1,e_2,e_3,e_4,e_5\}.$ The Lie bracket of $\so(4)$ with respect to the basis \eqref{eq:so4-basis} is given in Table \ref{tab:so4-bracket}.
\begin{table}[h]
\begin{center}
\newcolumntype{M}{>{$}c<{$}}
\begin{tabular}{|M||M|M|M|M|M|M|}
\hline
[\cdot,\cdot] & e_1 & e_2   & e_3  & e_4  & e_5  & e_6  \\ \hline \hline
e_1           & 0   & -e_6  & e_5  & 0    & -e_3 & e_2  \\ \hline
e_2           & e_6 & 0     & 0    & e_5  & -e_4 & -e_1 \\ \hline
e_3           & -e_5& 0     & 0    & -e_6 & e_1  & e_4  \\ \hline
e_4           & 0   & -e_5  & e_6  & 0    & e_2  & -e_3 \\ \hline
e_5           & e_3 & e_4   & -e_1 & -e_2 & 0    & 0    \\ \hline
e_6           & -e_2& e_1   & -e_4 & e_3  & 0    & 0    \\ \hline
\end{tabular}
\vspace{.5cm}
\caption{\label{tab:so4-bracket} Lie bracket of $\so(4)$ with respect to the basis \eqref{eq:so4-basis}.}
\end{center}
\end{table}
On $\so(4)$ we use the inner product 
$$
Q(a,b)=\frac12\operatorname{tr}(ab^T).
$$
With respect to $Q$, the basis \eqref{eq:so4-basis} is orthonormal. We set
$ X_i=e_i,$  $i=1,\dots,5, $ and denote by $\{X^1,\dots,X^5\}$ the dual coframe.  In order to keep the notation consistent with the adapted frame \eqref{eq:adaptedFrame}, we consider the $\Ad(SO(2))$-invariant splitting
$$
\fm=\fm_1\oplus\fm_2\oplus\fm_3,
$$
where
$\fm_1=\Span\{X_1,X_2\},$  $\fm_2=\Span\{X_3,X_4\}$ 
and $\fm_3=\Span\{X_5\}$. The factor $\fm_1\oplus\fm_2$ identifies with the contact distribution, while $\fm_3$ corresponds to the Reeb direction, thus $\xi=X_5$. From Table \ref{tab:so4-bracket}, the nonzero $\fm$-components of brackets between basis elements of $\fm$ are
\begin{equation}\label{eq:so4-brackets-m}
\begin{aligned}
[X_1,X_3]_{\fm}&=X_5, & [X_2,X_4]_{\fm}&=X_5, & [X_1,X_5]_{\fm}&=-X_3,\\
[X_2,X_5]_{\fm}&=-X_4,& [X_3,X_5]_{\fm}&=X_1, & [X_4,X_5]_{\fm}&=X_2.
\end{aligned}
\end{equation}
The nonzero $\fh$-components of brackets between basis elements of $\fm$ are
\begin{equation}\label{eq:so4-brackets-h}
[X_1,X_2]_{\fh}=-e_6, \qquad [X_3,X_4]_{\fh}=-e_6.
\end{equation}
Finally, the infinitesimal isotropy action of $\fh$ on $\fm$ is given by
\begin{equation}\label{eq:so4-brackets-hm}
\begin{aligned}
[e_6,X_1]&=-X_2, & [e_6,X_2]&=X_1, & [e_6,X_3]&=-X_4,&
[e_6,X_4]&=X_3, & [e_6,X_5]&=0.
\end{aligned}
\end{equation}
Here $[\cdot,\cdot]_{\fm}$ and $[\cdot,\cdot]_{\fh}$ denote the projections of the Lie bracket in $\so(4)=\fh\oplus\fm$ onto $\fm$ and $\fh$, respectively. The homogeneous contact form is given by  $\eta=X^5.$ For invariant $1$-forms on a reductive homogeneous space, we have
$$
d\alpha(X,Y)=-\alpha([X,Y]_{\fm}).
$$
Therefore, using \eqref{eq:so4-brackets-m}, the exterior derivative of $\eta$ is
$$
d\eta=-(X^{13}+X^{24}), \qquad \text{where } X^{ij}=X^i\wedge X^j.
$$ 
The endomorphism $\Phi:\fm\to\fm$ is defined by
$$
\Phi X_1=X_3,\qquad \Phi X_3=-X_1,\qquad
\Phi X_2=X_4,\qquad \Phi X_4=-X_2,\qquad \Phi X_5=0.
$$
This convention is compatible with $d\eta(X,Y)=g(X,\Phi Y)$. Moreover, the orientation determined by the ordered coframe $\{X^1,X^2,X^3,X^4,X^5\}$ agrees with the orientation fixed in \eqref{eq:volume-adapted-frame}. Hence
$$
\vol_g=X^{12345}=-\frac12 d\eta\wedge d\eta\wedge\eta.
$$ 
The homogeneous structure above coincides with the canonical Sasakian
structure induced on the unit tangent sphere bundle of the round sphere.
%---------------------------------------------------------------------
\subsection{Homogeneous principal bundles and invariant connections}
\label{sec:10}
%---------------------------------------------------------------------
We recall some standard results about homogeneous principal bundles, following
\cite[Section~1.4]{cap2009}. Let $H\subset G$ be a closed subgroup of a Lie
group $G$, and let $M=G/H$ be the corresponding homogeneous space. We denote by
$$
L:G\times M\longrightarrow M
$$
the canonical left action and by $p:G\longrightarrow G/H$ the canonical $H$-principal bundle.
\begin{definition}
A \emph{homogeneous principal $K$-bundle} over $M=G/H$ is a principal
$K$-bundle $\pi:P\to M$ together with a left $G$-action $\widetilde L:G\times P\longrightarrow P$ lifting the action $L$ on $M$, that is,
$$
\pi(\widetilde L_g(y))=L_g(\pi(y)),
$$
and such that, for each $g\in G$, the map
$\widetilde L_g:P\to P$ is an automorphism of principal $K$-bundles.
\end{definition}
Homogeneous principal bundles are characterized as follows.
\begin{lemma}\cite[Lemma~1.4.5]{cap2009}
\label{lem:homogeneous-principal-bundles}
Let $G$ and $K$ be Lie groups and let $H\subset G$ be a closed Lie subgroup.
If $P\to G/H$ is a homogeneous principal $K$-bundle, then there exists a smooth
homomorphism $\lambda:H\rightarrow K$ such that $P\cong G\times_\lambda K.$ Moreover, two homomorphisms $\lambda,\widetilde\lambda:H\to K$ define  isomorphic homogeneous principal $K$-bundles if and only if they are conjugate
by an element of $K$.
\end{lemma}
In the present case, $T_1S^3\cong SO(4)/SO(2),$ so homogeneous principal $K$-bundles over $T_1S^3$ are determined, up to isomorphism, by conjugacy classes of homomorphisms $\lambda:SO(2)\longrightarrow K.$ The corresponding homogeneous bundle is
$$
P_\lambda=SO(4)\times_{\lambda}K\longrightarrow SO(4)/SO(2).
$$
Since $SO(2)\cong S^1$ is abelian, the image of any homomorphism
$\lambda:SO(2)\to K$ is an abelian compact subgroup of $K$. If $K$ is compact
and connected, then every such homomorphism is conjugate to one with image
contained in a maximal torus $T_K\subset K$. Equivalently, such homomorphisms
are described by integral one-parameter subgroups of $T_K$, modulo conjugation.
In the concrete cases considered below, this description becomes explicit for
$K=U(1)$ and $K=SO(3)$.

For a homogeneous principal $K$-bundle $P\to G/H$, a connection
$A\in\Omega^1(P,\mathfrak k)$ is called \emph{invariant} if
$$
\widetilde L_g^\ast A=A \qquad \text{for every } g\in G.
$$
We shall use Wang's correspondence for invariant connections on homogeneous
principal bundles \cite{wang1958invariant}, in the form presented in
\cite[Theorem~1.4.5]{cap2009}.
\begin{lemma}[Wang's theorem]\label{lem:wang-theorem-T1S3}
Let $P_\lambda=G\times_{\lambda}K\longrightarrow G/H$ be the homogeneous principal $K$-bundle induced by the homomorphism $\lambda:H\to K$. Then invariant principal connections on $P_\lambda$ are in bijective correspondence with linear maps $\alpha:\mathfrak g\longrightarrow\mathfrak k$ satisfying
\begin{itemize}
\item[(i)] $\alpha_{\vert\mathfrak h}=\lambda_\ast:\mathfrak h\to\mathfrak k$;

\item[(ii)] $\alpha\circ\Ad(h)=\Ad(\lambda(h))\circ\alpha$ for all $h\in H$.
\end{itemize}
\end{lemma}
For later use, we also recall the curvature formula associated with Wang's
description. Let $\mathfrak g=\mathfrak h\oplus\mathfrak m$ be the fixed
reductive decomposition and let $\Lambda=\alpha_{\vert\fm}$. Then the curvature
of the invariant connection determined by $\alpha$ is given at the origin by
\begin{equation}\label{eq:homogeneous-invariant-curvature}
F_\Lambda(X,Y)=[\Lambda(X),\Lambda(Y)]-\Lambda([X,Y]_{\fm})-\lambda_\ast([X,Y]_{\fh}),\quad X,Y\in\fm.
\end{equation}
For our reductive decomposition
$$
\so(4)=\fh\oplus\fm, \qquad \fh=\Span\{e_6\},
$$
we write $\Lambda:=\alpha_{\vert\fm}.$ Since $SO(2)$ is connected and generated by $\exp(\theta e_6)$, condition  \textup{(ii)} can be written infinitesimally as
\begin{equation}\label{eq:wang-infinitesimal-general}
\Lambda([e_6,X])= [(\lambda_\ast e_6),\Lambda(X)], \qquad X\in\fm.
\end{equation}
This infinitesimal form will be used together with the isotropy action
\eqref{eq:so4-brackets-hm} to describe invariant connections explicitly.
%--------------------------------------------------------
\subsection{Homogeneous $U(1)$-bundles}\label{sec:11}
%--------------------------------------------------------
We now specialize the previous discussion to the case $K=U(1)$. Since $SO(2)\cong S^1$, every homomorphism  $\lambda:SO(2)\longrightarrow U(1)$ is of the form
\begin{equation}\label{eq:U1-homomorphism-lambda-m}
\lambda_m(\exp(\theta e_6))=\exp(im\theta), \qquad m\in\mathbb Z.
\end{equation}
Since $U(1)$ is abelian,   conjugation in $U(1)$ is trivial. Hence $\lambda_m$ and $\lambda_n$ are conjugate if and only if $m=n$. Therefore,
homogeneous principal $U(1)$-bundles over
$$
T_1S^3\cong SO(4)/SO(2)
$$
are classified by integers $m\in\mathbb Z$. We denote the bundle corresponding to $m$ by
$$
P_m=SO(4)\times_{\lambda_m}U(1) \longrightarrow SO(4)/SO(2).
$$
The integer $m$ determines the first Chern class of the associated homogeneous line bundle, up to the sign convention chosen in the identification $SO(2)\cong U(1)$. Let $\fu(1)=i\mathbb R$. Differentiating
\eqref{eq:U1-homomorphism-lambda-m} at the identity gives
\begin{equation}\label{eq:U1-lambda-differential}
(\lambda_m)_*(e_6)=im.
\end{equation}
Let $A$ be an invariant connection on $P_m$, and let
$\alpha:\so(4)\to\fu(1)$ be the corresponding Wang map. Then
$$
\alpha_{\vert\fh}=(\lambda_m)_*,
\qquad
\Lambda:=\alpha_{\vert\fm}.
$$
Since $U(1)$ is abelian, the adjoint action on $\fu(1)$ is trivial. Therefore
the infinitesimal Wang condition \eqref{eq:wang-infinitesimal-general} becomes
$$
\Lambda([e_6,X])=0,
\qquad
X\in\fm.
$$
Using the isotropy action \eqref{eq:so4-brackets-hm}, we obtain
$$
\Lambda(X_1)=\Lambda(X_2)=\Lambda(X_3)=\Lambda(X_4)=0.
$$
On the trivial summand $\fm_3=\Span\{X_5\}$, there is one free real parameter:
$$
\Lambda(X_5)=i\mu, \qquad \mu\in\mathbb R.
$$
Since $U(1)$ is abelian, the first term of \eqref{eq:homogeneous-invariant-curvature} vanishes. Hence
$$
F_\Lambda(X,Y) = -\Lambda([X,Y]_{\fm}) - (\lambda_m)_\ast([X,Y]_{\fh}).
$$
Using \eqref{eq:so4-brackets-m}, \eqref{eq:so4-brackets-h}, and
\eqref{eq:U1-lambda-differential}, the components which enter the local
monopole system are
\begin{equation}\label{eq:U1-relevant-curvature-components}
F_{13}=F_{24}=-i\mu,\quad F_{15}=F_{25}=F_{35}=F_{45}=0.
\end{equation}
The nonzero isotropy components are
$$
F_{12}=im, \qquad F_{34}=im.
$$
We summarize this computation as follows.
\begin{proposition}\label{prop:U1-invariant-monopoles}
Let $P_m=SO(4)\times_{\lambda_m}U(1)\longrightarrow T_1S^3$ be the homogeneous $U(1)$-bundle determined by  \eqref{eq:U1-homomorphism-lambda-m}, with $m\in\mathbb Z$. For $\mu\in\mathbb R$, let $\alpha_{m,\mu}:\so(4)\to\fu(1)$ be the Wang map defined by
$$
\alpha_{m,\mu}(e_6)=im, \qquad \alpha_{m,\mu}(X_5)=i\mu, \qquad \alpha_{m,\mu}(X_j)=0,\quad j=1,2,3,4.
$$
Let $A_{m,\mu}$ be the invariant connection on $P_m$ determined by
$\alpha_{m,\mu}$. Then, for every invariant Higgs field
$\phi\in i\mathbb R$, the pair $(A_{m,\mu},\phi)$ satisfies the contact
monopole equation if and only if $\mu=0.$
\end{proposition}
\begin{proof}
Since $U(1)$ is abelian, the adjoint representation is trivial. Hence every
invariant Higgs field $\phi\in i\mathbb R$ is constant and satisfies
$$
\nabla_{A_{m,\mu}}\phi=0.
$$
By Corollary \ref{cor:parallel-higgs-components}, the contact monopole equation
is equivalent to
$$
F_{15}=F_{25}=F_{35}=F_{45}=F_{13}+F_{24}=0.
$$
Using the curvature components computed in
\eqref{eq:U1-relevant-curvature-components}, we have
$$
F_{15}=F_{25}=F_{35}=F_{45}=0,
\qquad
F_{13}+F_{24}=-2i\mu.
$$
Therefore the contact monopole equation holds if and only if $\mu=0$.
\end{proof}
\begin{remark}
Proposition \ref{prop:pure-pullback-rigidity} shows that pullbacks from the base $S^3$ solve the contact monopole equation \eqref{eq:monopole 5} only in the flat and parallel case. Thus, from the point of view of the fibration
$ 
S^2\longrightarrow T_1S^3\longrightarrow S^3,
$ 
naive pullback constructions do not produce nontrivial contact monopoles.  By contrast, the homogeneous $U(1)$-bundles
$$
P_m=SO(4)\times_{\lambda_m}U(1)\longrightarrow SO(4)/SO(2)
$$
provide genuine $5$-dimensional examples of contact monopoles for $m\neq0$.
\end{remark}
%---------------------------------------------------------------------
\subsection{Homogeneous $SO(3)$-bundles}
\label{subsec:homogeneous-SO3-bundles}
%---------------------------------------------------------------------

We now consider the first non-abelian case. According to Lemma
\ref{lem:homogeneous-principal-bundles}, homogeneous principal $SO(3)$-bundles
over $T_1S^3$ are determined, up to isomorphism, by conjugacy classes of
homomorphisms
$$
\lambda:SO(2)\longrightarrow SO(3).
$$
We fix a basis $\{T_1,T_2,T_3\}$ of $\so(3)$ satisfying
\begin{equation}\label{eq:so3-brackets}
[T_1,T_2]=T_3,\qquad
[T_2,T_3]=T_1,\qquad
[T_3,T_1]=T_2.
\end{equation}
Equivalently, one may take
$ 
T_1=
\begin{pmatrix}
0&0&0\\
0&0&-1\\
0&1&0
\end{pmatrix},
$ $
T_2=
\begin{pmatrix}
0&0&1\\
0&0&0\\
-1&0&0
\end{pmatrix} 
$ and $
T_3=
\begin{pmatrix}
0&-1&0\\
1&0&0\\
0&0&0
\end{pmatrix}. $
The subalgebra $\ft=\Span\{T_3\}$ is a maximal torus subalgebra of $\so(3)$. Since $SO(2)\cong S^1$ is abelian, the image of any homomorphism
$$
\lambda:SO(2)\to SO(3)
$$ 
is contained, up to conjugation, in a maximal torus of
$SO(3)$. Hence every such homomorphism is conjugate to one of the form
\begin{equation}\label{eq:SO3-homomorphism-lambda-m}
\lambda_m(\exp(\theta e_6)) = \exp(m\theta T_3), \quad m\in\mathbb Z.
\end{equation}
The integer $m$ is the weight of the homomorphism. The homomorphisms
$\lambda_m$ and $\lambda_{-m}$ correspond to rotations around the same axis
with opposite orientations. They are conjugate in $SO(3)$, because an element
of $SO(3)$ can send the chosen axis to its opposite. Therefore, up to
conjugation, the homomorphisms $SO(2)\to SO(3)$ are represented by
$$
m\in\mathbb Z_{\geq0}.
$$
Thus the homogeneous principal $SO(3)$-bundles over $T_1S^3$ are given by
\begin{equation}\label{eq:SO3-homogeneous-bundles}
P_m = SO(4)\times_{\lambda_m}SO(3) \rightarrow SO(4)/SO(2), \quad m\in\mathbb Z_{\geq0}.
\end{equation}
In this subsection, $P_m$ denotes the homogeneous principal $SO(3)$-bundle
associated with $\lambda_m$. Differentiating
\eqref{eq:SO3-homomorphism-lambda-m} at the identity gives
\begin{equation}\label{eq:SO3-lambda-differential}
(\lambda_m)_\ast(e_6)=mT_3.
\end{equation}
The case $m=0$ corresponds to the trivial homomorphism. The case $m=1$ is the
standard inclusion of the isotropy circle into a maximal torus of $SO(3)$.
%----------------------------------------------------------------------------
\subsection{Invariant $\SO(3)$-connections and contact monopoles}
\label{subsec:invariant-SO3-connections}
%---------------------------------------------------------------------------- 
Let $P_m$ be the homogeneous $SO(3)$-principal bundle
\eqref{eq:SO3-homogeneous-bundles}. By Wang's theorem
(cf. Lemma \ref{lem:wang-theorem-T1S3}), an invariant connection on $P_m$ is
determined by a linear map
$$
\alpha:\so(4)\longrightarrow\so(3)
$$
satisfying

$$
\alpha_{\vert\fh}=(\lambda_m)_\ast
$$
and
$$
\alpha\circ\Ad(h)=\Ad(\lambda_m(h))\circ\alpha,
\qquad
h\in SO(2).
$$
Writing
$$
\Lambda:=\alpha_{\vert\fm},
$$
and using \eqref{eq:SO3-lambda-differential}, the infinitesimal Wang condition
\eqref{eq:wang-infinitesimal-general} becomes
\begin{equation}\label{eq:SO3-infinitesimal-equivariance}
\Lambda([e_6,X])=[mT_3,\Lambda(X)],
\qquad
X\in\fm.
\end{equation}
Together with the isotropy action \eqref{eq:so4-brackets-hm}, this determines
the possible invariant connections.

We first record the form of the Wang maps. If $m\neq1$, the weight-one
summands $\fm_1$ and $\fm_2$ cannot map equivariantly to the weight-$m$
summand of $\so(3)$. This follows directly from
\eqref{eq:SO3-infinitesimal-equivariance}. For instance, let
$$
A=\Lambda(X_1),
\qquad
B=\Lambda(X_2).
$$
Using \eqref{eq:so4-brackets-hm}, we obtain
$$
-B=[mT_3,A],
\qquad
A=[mT_3,B].
$$
Writing $A=aT_1+bT_2+cT_3$, one gets
$$
c=0,
\qquad
(1-m^2)a=(1-m^2)b=0.
$$
Hence, since $m\geq0$,  for $m\neq1$ we have 
$$
\Lambda(X_1)=\Lambda(X_2)=0.
$$
The same argument gives
$$
\Lambda(X_3)=\Lambda(X_4)=0.
$$
Finally, since $[e_6,X_5]=0$, the equivariance condition \eqref{eq:SO3-infinitesimal-equivariance} gives
$$
[mT_3,\Lambda(X_5)]=0.
$$
Thus, for $m\geq1$, this means that $\Lambda(X_5)$ lies in the centralizer of
$T_3$ in $\so(3)$, hence 
$$
\Lambda(X_5)=\mu T_3,\quad \mu\in\mathbb R,
$$
whereas for $m=0$ one has
$$
\Lambda(X_5)=\nu,\quad\nu\in\so(3).
$$

For $m=1$, the infinitesimal equivariance condition \eqref{eq:SO3-infinitesimal-equivariance} becomes
$$
\Lambda([e_6,X])=[T_3,\Lambda(X)].
$$
Since the action of $e_6$ rotates the planes $\fm_1$ and $\fm_2$ with the same
weight as the adjoint action of $T_3$ on $\Span\{T_1,T_2\}$, equivariant
nonzero maps from $\fm_1$ and $\fm_2$ to $\Span\{T_1,T_2\}$ may occur. 
Indeed, solving the equivariance equations on $\fm_1$ gives
$$
\Lambda(X_1)=aT_1+bT_2 \quad \text{and}  \quad \Lambda(X_2)=bT_1-aT_2;
$$ 
the same argument applied to $\fm_2$ gives the components along $X_3,X_4$, while the condition
$[T_3,\Lambda(X_5)]=0$ forces $\Lambda(X_5)$ to lie in $\Span\{T_3\}$.
\begin{equation}\label{eq:SO3-Lambda-m1}
\begin{aligned}
\Lambda(X_1)&=aT_1+bT_2,&\Lambda(X_2)&=bT_1-aT_2,\\
\Lambda(X_3)&=cT_1+dT_2,&\Lambda(X_4)&=dT_1-cT_2,\\
\Lambda(X_5)&=\mu T_3.
\end{aligned}
\end{equation}
An invariant Higgs field is determined by an element $\phi_0\in\so(3)$ fixed by the isotropy action:
$$
\Ad(\lambda_m(h))\phi_0=\phi_0,\quad h\in SO(2).
$$
Thus, for $m\geq1$, we have $\phi_0=\tau T_3, $  $\tau\in\mathbb R.$
For $m=0$, the isotropy action on $\so(3)$ is trivial, and therefore
$\phi_0$ may be any element of $\so(3)$.

The curvature of an invariant connection determined by $\Lambda$ is given at
the origin by \eqref{eq:homogeneous-invariant-curvature}.  We first treat  the cases $m\neq1$.
\begin{lemma}\label{lem:SO3-m-not-one}
Let $P_m\to T_1S^3$ be the homogeneous $SO(3)$-bundle \eqref{eq:SO3-homogeneous-bundles} associated with $\lambda_m$, with $m\neq1$. Then the invariant contact monopole equation forces $\Lambda=0.$ Consequently, the corresponding invariant connection reduces to the maximal torus $SO(2)=\exp(\Span\{T_3\})\subset SO(3).$ More precisely:
\begin{itemize}
    \item if $m=0$, the connection is flat; 
    \item if $m\geq2$, the connection is not necessarily flat; however, the 
          remaining isotropy component of the Wang map is
          $$
            (\lambda_m)_*(e_6)=mT_3,
          $$
         and the curvature takes values in the abelian subalgebra
          $\Span\{T_3\}\subset\so(3)$.  
\end{itemize}
In particular, the cases $m\neq1$ do not produce genuinely non-abelian homogeneous contact monopoles.
\end{lemma}
\begin{proof}
For $m\neq1$, the equivariance condition gives
$$
\Lambda(X_1)=\Lambda(X_2)=\Lambda(X_3)=\Lambda(X_4)=0.
$$

If $m=0$, then 
$$
\Lambda(X_5)=\nu, \quad \nu\in\so(3),
$$ 
and the invariant Higgs field $\phi_0$ is arbitrary. Using the curvature formula \eqref{eq:homogeneous-invariant-curvature} and the $\fm$-bracket identities \eqref{eq:so4-brackets-m}, we obtain
$$
F_{13}=F_\Lambda(X_1,X_3)=-\Lambda([X_1,X_3]_{\fm})=-\Lambda(X_5)=-\nu,
$$
and similarly $F_{24}=-\nu.$ Moreover,
$$
F_{15}=F_{25}=F_{35}=F_{45}=0.
$$
For an invariant Higgs field and an invariant connection, the covariant
derivative in a homogeneous direction is
$$
\nabla_{A,X}\phi_0=[\Lambda(X),\phi_0].
$$
Hence, since $\xi=X_5$, the last equation of
\eqref{eq:local-monopole-system} gives
$$
[\nu,\phi_0]=F_{24}+F_{13}=-2\nu.
$$
Taking the inner product with $\nu$ and using $\Ad$-invariance, we obtain
$$
0=\langle[\nu,\phi_0],\nu\rangle=-2|\nu|^2.
$$
Thus $\nu=0$, and therefore $\Lambda=0$ in this case.

If $m\geq2$, then
$$
\Lambda(X_5)=\mu T_3,
\qquad
\phi_0=\tau T_3.
$$
Again, using \eqref{eq:homogeneous-invariant-curvature} and
\eqref{eq:so4-brackets-m}, we obtain
$$
F_{13}=F_{24}=-\mu T_3,
\qquad
F_{15}=F_{25}=F_{35}=F_{45}=0.
$$
Since $[\Lambda(X_5),\phi_0]=[\mu T_3,\tau T_3]=0,$ the last equation of \eqref{eq:local-monopole-system} gives
$$
0=F_{24}+F_{13}=-2\mu T_3.
$$
Therefore $\mu=0$. Thus, for $m\neq1$, the invariant monopole equation forces the $\fm$-component
of the Wang map to vanish. If $m=0$, the connection is flat. If $m\geq2$, the
remaining isotropy component of the Wang map is
$$
(\lambda_m)_*(e_6)=mT_3,
$$
so the corresponding connection reduces to the maximal torus $\exp(\Span\{T_3\})\subset SO(3).$
This proves the claim.
\end{proof}
We are therefore left with the case $m=1$, which is the only case where
nonzero components of $\Lambda$ along $\fm_1\oplus\fm_2$ may occur.

Using \eqref{eq:SO3-Lambda-m1}, \eqref{eq:so4-brackets-m}, and
\eqref{eq:homogeneous-invariant-curvature}, we compute the curvature components which
enter the local monopole system \eqref{eq:local-monopole-system}. For example,
\begin{align*}
F_{35} 
&= F_\Lambda(X_3,X_5) =
[\Lambda(X_3),\Lambda(X_5)] - \Lambda([X_3,X_5]_{\fm})- (\lambda_1)_*([X_3,X_5]_{\fh}).
\end{align*}
By \eqref{eq:so4-brackets-m}, we have
$[X_3,X_5]_{\fm}=X_1, $ and this bracket has no $\fh$-component. Hence
\begin{align*}
F_{35}&= [\Lambda(X_3),\Lambda(X_5)] - \Lambda(X_1)
=[cT_1+dT_2,\mu T_3] - (aT_1+bT_2).
\end{align*}
Using the bracket relations \eqref{eq:so3-brackets}, namely
$$
[T_1,T_3]=-T_2,\quad [T_2,T_3]=T_1,
$$
we obtain
\begin{align*}
[cT_1+dT_2,\mu T_3] &= \mu c[T_1,T_3]+\mu d[T_2,T_3] = -\mu cT_2+\mu dT_1.
\end{align*}
Therefore
$$
F_{35}
=
(\mu d-a)T_1+(-\mu c-b)T_2.
$$
The remaining components are computed in the same way. Thus the curvature
components entering \eqref{eq:local-monopole-system} are
\begin{equation}\label{eq:SO3-m1-relevant-curvature}
\begin{alignedat}{2}
F_{35} &= (\mu d-a)T_1+(-\mu c-b)T_2,
&\quad
F_{45} &= (-\mu c-b)T_1+(a-\mu d)T_2,\\
F_{15} &= (\mu b+c)T_1+(d-\mu a)T_2,
&\quad
F_{25} &= (d-\mu a)T_1+(-c-\mu b)T_2,\\
F_{24}+F_{13} &= 2(ad-bc-\mu)T_3. &&
\end{alignedat}
\end{equation}
Moreover, since $\phi_0=\tau T_3$ and the Higgs field is invariant, for
$X\in\fm$ we have
$$
\nabla_{A,X}\phi_0=[\Lambda(X),\phi_0].
$$
For example, using \eqref{eq:SO3-Lambda-m1} and \eqref{eq:so3-brackets},
\begin{align*}
\nabla_{A,X_1}\phi_0
&= [\Lambda(X_1),\phi_0] = [aT_1+bT_2,\tau T_3]= \tau a[T_1,T_3]+\tau b[T_2,T_3]\\
&= -\tau aT_2+\tau bT_1 = \tau(bT_1-aT_2).
\end{align*}
The remaining components are computed in the same way. Thus
\begin{equation}\label{eq:SO3-m1-higgs-derivatives}
\begin{alignedat}{2}
\nabla_{A,X_1}\phi_0 &= \tau(bT_1-aT_2),
&\quad \nabla_{A,X_2}\phi_0 &= -\tau(aT_1+bT_2),\\
\nabla_{A,X_3}\phi_0 &= \tau(dT_1-cT_2),
&\quad \nabla_{A,X_4}\phi_0 &= -\tau(cT_1+dT_2),\\
\nabla_{A,X_5}\phi_0 &= 0. &&
\end{alignedat}
\end{equation}
Therefore, by Proposition \ref{prop:local-monopole-system}, the invariant
contact monopole equation is obtained by comparing
\eqref{eq:SO3-m1-relevant-curvature} with
\eqref{eq:SO3-m1-higgs-derivatives}. More explicitly, the first equation
$\nabla_{A,X_1}\phi_0=F_{35}$ gives
$$
\tau(bT_1-aT_2) = (\mu d-a)T_1+(-\mu c-b)T_2.
$$
Equating the $T_1$ and $T_2$ components, we obtain
$\mu d-a=\tau b$ and $\mu c+b=\tau a.$ The second equation $\nabla_{A,X_2}\phi_0=F_{45}$  gives the same two relations. Similarly, the equations
$$
\nabla_{A,X_3}\phi_0=-F_{15}, \qquad \nabla_{A,X_4}\phi_0=-F_{25}
$$
give $\tau d=-\mu b-c$ and $\tau c=d-\mu a.$ Finally, since $\nabla_{A,X_5}\phi_0=0$ and
$$
F_{24}+F_{13}=2(ad-bc-\mu)T_3,
$$
the last equation gives $\mu=ad-bc.$  Consequently, the invariant contact monopole equation for $m=1$ is equivalent to the following algebraic system in the real parameters $a,b,c,d,\mu,\tau\in\mathbb R,$
where $a,b,c,d,\mu$ determine the invariant connection through the Wang map
\eqref{eq:SO3-Lambda-m1}, and $\tau$ determines the invariant Higgs field
$\phi_0=\tau T_3$:
\begin{equation}\label{eq:SO3-m1-monopole-system}
\left\{
\begin{alignedat}{3}
\tau b &=\mu d-a,        &\qquad \tau a&=\mu c+b, &\quad\mu&=ad-bc\\
\tau d&=-\mu b-c,      &\qquad \tau c&=d-\mu a. 
\end{alignedat}
\right.
\end{equation}
\begin{proposition}\label{prop:SO3-m1-monopoles}
Let $P_1\to T_1S^3$ be the homogeneous $SO(3)$-bundle \eqref{eq:SO3-homogeneous-bundles} corresponding to $m=1$. Let $A_{a,b,c,d,\mu}$ be the invariant connection determined by the Wang map \eqref{eq:SO3-Lambda-m1}, where $a,b,c,d,\mu\in\mathbb R$. Let $\phi_\tau$ be the invariant Higgs field  determined by $\phi_\tau=\tau T_3,$ $\tau\in\mathbb R.$
Then the invariant pair $(A_{a,b,c,d,\mu},\phi_\tau)$ satisfies the contact
monopole equation if and only if one of the following alternatives holds:
\begin{enumerate}
    \item the toral branch:
    $ a=b=c=d=\mu=0,$  $\tau\in\mathbb R;$

    \item the genuinely non-abelian branch:
    $\tau=0,$ $\mu=\pm1,$ $c=-\mu b,$ $d=\mu a$ and $a^2+b^2=1.$
\end{enumerate}
In particular, the case $m=1$ produces invariant contact monopoles whose Wang
map has nonzero components along the contact distribution $\fm_1\oplus\fm_2$.
\end{proposition}
\begin{proof}
The invariant monopole equation is equivalent to the algebraic system
\eqref{eq:SO3-m1-monopole-system}. We solve this system by considering the two
possibilities $\mu=0$ and $\mu\neq0$.
First suppose that $\mu=0$. Then the first two equations of
\eqref{eq:SO3-m1-monopole-system} give
$$ 
-a=\tau b, \quad b=\tau a.
$$
Hence $a=-\tau b$ and $b=-\tau^2 b$, so
$$
(1+\tau^2)b=0.
$$
Thus $b=0$, and consequently $a=0$. Similarly, the third and fourth equations
give
$$
\tau d=-c,
\qquad
\tau c=d.
$$
Substituting $d=\tau c$ into the first of these equations gives
$$
\tau^2c=-c,
$$
hence $c=0$ and then $d=0$.Therefore $a=b=c=d=\mu=0$, while $\tau\in\mathbb R$ is arbitrary. This gives the toral branch.

Now suppose that $\mu\neq0$. From the third and fourth equations of
\eqref{eq:SO3-m1-monopole-system}, we obtain
$$
c=-\frac{\mu}{1+\tau^2}(b+\tau a),
\qquad
d=\frac{\mu}{1+\tau^2}(a-\tau b).
$$
Using the last equation of \eqref{eq:SO3-m1-monopole-system}, namely
$\mu=ad-bc$, and since $\mu\neq0$, it follows that
\begin{equation}\label{eq:SO3-radius-relation}
a^2+b^2=1+\tau^2.
\end{equation}
Substituting the above expressions for $c$ and $d$ into the first two equations
of \eqref{eq:SO3-m1-monopole-system} gives
\begin{align*}
(\mu^2-1-\tau^2)a-\tau(\mu^2+1+\tau^2)b&=0,\\
-\tau(\mu^2+1+\tau^2)a-(\mu^2-1-\tau^2)b&=0.
\end{align*}
By \eqref{eq:SO3-radius-relation}, we have $(a,b)\neq(0,0)$. Hence the
determinant of this homogeneous linear system must vanish, that is,
$$
(\mu^2-1-\tau^2)^2+\tau^2(\mu^2+1+\tau^2)^2=0.
$$
Since both terms are nonnegative, each one must vanish. Moreover,
$$
\mu^2+1+\tau^2>0,
$$
and therefore the second term gives $\tau=0$. Substituting $\tau=0$ into the
first term gives
$$
(\mu^2-1)^2=0,
$$
hence $\mu^2=1$. Thus $\mu=\pm1$. With $\tau=0$, the expressions for $c$ and $d$ reduce to $c=-\mu b$ and $d=\mu a$, and \eqref{eq:SO3-radius-relation} becomes $a^2+b^2=1$. This gives the genuinely non-abelian branch and completes the proof.
\end{proof}

%% Loading bibliography style file
\bibliographystyle{unsrtnat}

% Loading bibliography database
\bibliography{refs}
%\vskip3pt 
\end{document}